\documentclass{amsart}
\usepackage{latexsym,amscd,amssymb,amsopn,amsthm,amsfonts,amsmath}
\usepackage{color}
\copyrightinfo{2010}{American Mathematical Society}
\newtheorem{theorem}{Theorem}[section]
\newtheorem{lemma}[theorem]{Lemma}
\newtheorem{corollary}[theorem]{Corollary}

\theoremstyle{definition}

\newtheorem{problem}[theorem]{Problem}

\theoremstyle{remark}

\numberwithin{equation}{section}

\begin{document}

\title{On the ill-posed Cauchy problem  for the polyharmonic heat equation}

\begin{abstract}
We consider 
the ill-posed Cauchy problem for the polyharmonic heat equation on  recovering 
a function, satisfying the equation $(\partial _t + (- \Delta)^m) u=0$ in a cylindrical domain in the half-space ${\mathbb R}^n \times [0,+\infty)$, where $n\geq 1$,
$m\geq 1$ and $\Delta$ is the Laplace operator,  via its values and 
the values of its normal derivatives up to order $(2m-1)$ on a given part of the lateral 
surface of the cylinder. We obtain a  
Uniqueness Theorem for the problem and a criterion of its solvability in terms of the 
real-analytic continuation of parabolic potentials, associated with the Cauchy data. 
\end{abstract}

\author[I. Kurilenko]{Ilya Kurilenko}
\email{ilyakurq@gmail.com}

\author[A. Shlapunov]{Alexander Shlapunov}
\email{ashlapunov@sfu-kras.ru}

\address{Siberian Federal University,
         Institute of Mathematics and Computer Science,
         pr. Svobodnyi 79,
         660041 Krasnoyarsk,
         Russia}

\subjclass [2010] {Primary 35K25; Secondary 35R25}

\keywords{the polyharmonic heat equation, ill-posed problems, integral 
representation's method.}

\maketitle

\section*{Introduction}

In this short note we continue to investigate the ill-posed Cauchy problem for 
parabolic operators in various function spaces, see \cite{PuSh12}, \cite{PuSh15}
for the second order operators in the H\"older spaces, or \cite{MMT17}, \cite{KuSh2019}, \cite
{VKuSh2022} for the second order operators in the anisotropic Sobolev spaces. 
Actually the general schemes related to investigation of the ill-posed Cauchy problem for 
elliptic operators (see \cite{Lv1,LRSh,TiAr} for the second order operators or 
\cite{A,AKy} for the Cauchy-Riemann system in one and many complex variables or 
\cite{ShTaLMS,Tark36} for general elliptic operators with the unique continuation property) 
are still applicable in this new situation. 

In the present paper we concentrated our efforts on the solvability criterion of the  
ill-posed Cauchy problem for a simple class of Petrovsky  $2m$-parabolic partial 
differential operators 
\begin{equation} \label{eq.heat.0}
(\partial _t + (- \Delta)^m), 
\end{equation} 
 where $m\geq 1$ and $\Delta$ is the Laplace operator in ${\mathbb R}^n$, $n\geq 1$, 
that are often called polyharmonic heat 
operators, see \cite[Ch.2, \S 1]{eid}, \cite{rep+eid}.  Namely the problem consists of the  
recovering a function, satisfying the equation 
$(\partial _t + (- \Delta)^m) u =0$  
in a cylindrical domain in the half-space ${\mathbb R}^n \times [0,+\infty)$, 
via its values and the values of its normal derivatives up to order $(2m-1)$ on a 
given part of the lateral surface of the cylinder. The crucial difference between the 
heat equation (or the parabolic Lam\'e system) and the polyharmonic heat equation 
is the fact that the fundamental solution of the polyharmonic heat operator 
is given by a non-elementary function. The situation resembles somehow the 
matter with the fundamental solutions to the Helmholtz operator $\Delta + c^2_0$: 
for $n=3$ it is given by $\frac{-  e^{\pm \iota c_0|x|}}{4\pi |x|}$ (here $\iota $ is the 
imaginary unit) while for $n=2$ it is represented by the Hankel functions of the second kind 
(actually, some versions of the Bessel functions),  see, \cite[Ch. III, \S 11]{Vla}. Of 
course, it is not a surprise, because after an application of the Laplace  transform 
$ L$ with respect to the variable $t$ (if applicable) to \eqref{eq.heat.0}, 
one arrives at the parameter depending elliptic equation 
\begin{equation} \label{eq.heat.0.F}
(\iota \tau + (- \Delta)^m) L (u) =0, 
\end{equation} 
coinciding with the Helmholtz equation for $m=1$ regarding the generalized 
function $L (u)$ as an unknown and $\tau$ as a real parameter. Actually, this seemingly 
simple approach, reducing the parabolic equations to elliptic ones, is known for decades, 
see \cite{AV64}. It gives a lot of qualitative information on the connection 
between the corresponding solutions of the differential equations of different kinds. 
However one needs very delicate properties of the 
Laplace transform in order to obtain really useful formulas solving the parabolic 
problems with the use of elliptic theory, see for instance, \cite{MMT17} for 
the  heat equation and 
the related remark on properties of the Laplace transform \cite{ChLyTa}.
Thus, we will act in the framework of mentioned above scheme invented by L. Aizenberg 
and developed in \cite{ShTaLMS}.

\section{Preliminaries}

\noindent Let $\Omega$ be a bounded domain  in $n$-dimensional linear space 
${\mathbb R}^n$ with the coordinates $x=(x_1, \dots , x_n)$. As usual we denote by 
$\overline{\Omega}$ the closure of $\Omega$, and we denote  by $\partial\Omega$ its boundary.
In the sequel we assume that $\partial \Omega$ is piece-wise smooth. 
We denote by  $\Omega_T$ the bounded open cylinder $\Omega \times (0,T)$
in ${\mathbb R}^{n+1}$ with a positive altitude $T$. 
 Let also $\Gamma\subset 
\partial\Omega$ be a non empty connected relatively open subset 
of $\partial\Omega$. Then $\Gamma_T = \Gamma\times(0,\ T)$ and 
$\overline{\Gamma_T}=\overline{\Gamma} \times[0,\ T]$.

We consider the  functions over subsets in ${\mathbb R}^n$ and 
${\mathbb R}^{n+1}$.  As usual, for $s \in {\mathbb Z}_+$  we denote by $C^s(\Omega)$ the 
space of all $s$ times continuously differentiable functions in $\Omega$. 
Next, for a (relatively open) set $S \subset \partial \Omega$  denote by 
$C^{s}(\Omega\cup S)$ the set of such 
functions from the space $C^{s}(\Omega)$ that all their derivatives up to order  $s$ can be 
extended continuously onto $\Omega \cup S$. The standard topology of these  
metrizable spaces induces the uniform convergence on compact subsets in $\Omega \cup 
S$ together with all the partial derivatives up to order  $s$.  
We will also use the standard Banach H\"older spaces $C^{s} (\overline \Omega)$  and 
$C^{s,\lambda} (\overline \Omega)$  (cf. \cite{frid}, \cite[Ch.1, \S 1]{lad}, 
\cite{Kry96},  and the related metrizable spaces $C^{s,\lambda} (\Omega\cup S)$. 

Let also $L^p (\Omega)$, $p\geq 1$, be the Lebesgue spaces, 
$H^s (\Omega)$, $s\geq 0$, stand for the Sobolev spaces if $s\in \mathbb N$ 
and  for the Sobolev-Slobodetskii spaces if $s>0$, $s\not \in \mathbb N$. 

To investigate the polyharmonic heat equation we need also the anisotropic ($2m$-parabolic) 
spaces, see \cite[Ch. 1]{lad}, \cite[Ch. 8]{Kry96} for $m=1$ and \cite{eid} for $m\geq 1$. 
With this aim,  let $C^{2ms,s} (\Omega_T)$, $m\in {\mathbb N}$, stand for the set of all the 
continuous functions $u$ in  $\Omega_T$, having in $\Omega_T$ the continuous partial 
derivatives $\partial^j_t \partial^{\alpha}_x u$ with all the multi-indexes $(\alpha,j) 
\in {\mathbb Z}_+^{n} \times {\mathbb Z}_+$ satisfying $|\alpha|+2mj \leq 2ms$ where, as 
usual,  $|\alpha|=\sum_{j=1}^n \alpha_j$.
Similarly,  we denote by $C^{2ms+k,s} (\Omega_T)$   the set of continuous  
functions in  $\Omega_T$, such that all the partial derivatives 
$\partial^\beta u $ belong to $ C^{2ms,s} (\Omega_T)$ if  
$\beta \in {\mathbb Z}_+^{n}$ satisfies $|\beta|\leq k$, $k \in {\mathbb Z}_+$.  
Of course, it is natural to agree that  
$C^{2ms+0,s} (\Omega_T) =C^{2ms,s} (\Omega_T)$,  $C^{0,0} (\Omega_T) =C(\Omega_T)$ 
and $C^{0} (\Omega) =C(\Omega)$. 
We also denote by $C^{2ms+k,s}((\Omega\cup S)_T )$ the set of such 
functions $u$ from the space $C^{2ms+k,s}(\Omega_T)$ that their partial 
derivatives  $\partial_t^j
\partial^{\alpha+\beta}_x u$, $2mj+|\alpha|\leq 2ms$, $|\beta|\leq k$,  
can be extended continuously onto $(\Omega \cup S)_T$. 
The standard topology of these  
metrizable space induces the uniform convergence on compact subsets of $(\Omega \cup 
S)_T$ together with all the partial derivatives used in its definition  
(the cases $S=\emptyset$ and $S=\partial D$ are included). 

We  use also the anisotropic H\"older spaces (cf., \cite[Ch. 1]{lad}, 
\cite[Ch. 8]{Kry96}) for $m=1$ and \cite{eid} for $m\geq 1$. 
Let $C^{2ms+k,s,\lambda,\lambda/2} ((\Omega \cup S)_T )$ 
stand for the set of anisotropic H\"older  continuous functions with a 
power $\lambda$ over each compact subset of $(\Omega \cup S)_T $ 
together with all the partial derivatives $\partial^{\alpha+\beta}_x \partial^j_t u$ 
where $|\beta|\leq k$, $|\alpha|+2mj\leq 2ms$. 
Clearly, $C^{2ms+k,s,\lambda,\lambda/2}(\overline {\Omega_T})$ 
is a Banach space with the natural norm, 
see, for instance, \cite[Ch. 8]{Kry96} for $m=1$ and \cite{eid} for $m\geq 1$.  
In general, the space $C^{2ms+k,s,\lambda,\lambda/2} ((\Omega\cup S)_T)$ 
can be treated again as a metrizable space, generated by a system of seminorms associated
with a suitable exhaustion $\{\Omega_i\}_{i \in \mathbb N}$ 
of the set $\Omega\cup S$. 

In order to invoke the Hilbert space approach, we need anisotropic ($2m$-parabolic) 
Sobolev spaces $H^{2ms,s} (\Omega_T)$, $s \in  {\mathbb Z}_+$, see, \cite{lad,Kry08} for $m=1$ or \cite{eid} for $m\geq 1$, i.e. 
the set of all the measurable  
functions $u$ in  $\Omega_T$ such that all the generalized partial derivatives 
$\partial^j_t \partial^{\alpha}_x u$ with all the multi-indexes $(\alpha,j) 
\in {\mathbb Z}_+^{n} \times {\mathbb Z}_+$ satisfying $|\alpha|+2mj \leq 2ms$,
belong to the Lebesgue class $L^{2} (\Omega_T)$. This is the  Hilbert 
 space with the natural inner product $(u,v)_{H^{2ms,s} (\Omega_T)}$.
We also may define $H^{2ms,s} (\Omega_T)$ as the completion
of the space $C^{2ms,s} (\overline{\Omega_T})$ with respect to the 
norm $\|\cdot \|_{H^{2ms,s} (\Omega_T)}$ generated by the inner product 
$(u,v)_{H^{2ms,s} (\Omega_T)}$. For $s=0$ we have $H^{0,0} (\Omega_T) = L^{2} (\Omega_T)$. 

We also will use the so-called Bochner spaces 
of functions depending on $(x,t)$ from the strip  
$\mathbb{R}^n \times  [T_1,T_2]$.
Namely, for a Banach space $\mathcal B$ (for example, the space of functions 
on a subdomain of $\mathbb{R}^n$) and    $p \geq 1$, we denote by  
$L^p ([T_1,T_2],{\mathcal B})$ the Banach space of all the measurable mappings 
  $u : [T_1,T_2] \to {\mathcal B}$
with the finite norm  $\| u \|_{L^p ([T_1,T_2],{\mathcal B})}
 := \| \|  u (\cdot,t) \|_{\mathcal B} \|_{L^p ([T_1,T_2])}$, 
see, for instance, \cite[ch. \S 1.2]{Lion69}. 
The space $C ([T_1,T_2],{\mathcal B})$ is introduced with the use of the 
same scheme; this is the Banach space of all the continuous mappings
$u : [T_1,T_2] \to {\mathcal B}$ with the finite norm 
$
   \| u \|_{C ([T_1,T_2],{\mathcal B})}
 := \sup_{t \in [T_1,T_2]} \| u (\cdot,t) \|_{\mathcal B}$. 

Let now $\Delta=\sum\limits_{j=1}^n \partial^2 _{x_j,x_j}$ be the Laplace 
operator  in ${\mathbb R}^n$ and let 
${\mathcal L}_m= \partial_t +(-\Delta)^m $ 
stand for the polyharmonic heat operator  in ${\mathbb R}^{n+1}$.  
Of course, for $m=1$ it coincides the usual heat operator. 

Now let $\partial_ \nu=\sum\limits_{j=1}^n\nu_j \partial_{x_j}$ denote the derivative at 
the direction of the exterior unit normal vector $\nu=(\nu_1, ..., \nu_n)$ to the surface 
$\partial\Omega$. If $\partial\Omega \in C^{2m-1}$ then the higher order normal 
derivatives $\partial_ \nu^j$ are defined near $\partial \Omega$. We fix also a Dirichlet 
system $\{ B_j\}_{j=0}^{2m-1}$ of order $(2m-1)$  consisting of boundary differential  
operators with smooth coefficients near $\partial \Omega$,  i.e. ${\rm ord} B_j =j$ and 
for each $x\in \partial \Omega$ the characteristic polynomials $\sigma (B_j)(x,\zeta)$
related to the operators $B_j$ do not vanish for $\zeta =\nu (x)$. The sets 
$(1,\partial_ \nu, \partial_ \nu^2, \dots \partial_ \nu^{2m-1})$ and 
$(1,\partial_ \nu, \Delta, \partial_ \nu \Delta , \Delta^2 ,\dots \Delta^{m-1}, \partial_ \nu 
\Delta^{m-1} )$ are precisely the Dirichlet systems because $\sigma (\partial_\nu^j)(x,\nu(x))=
\sigma (\partial_\nu \Delta^j)(x,\nu(x)) = \sigma (\Delta^j)(x,\nu(x)) =1$ for each $j \in \mathbb N$.

We consider the Cauchy problem for the polyharmonic heat equation in 
the cylinder $\Omega_T$ in the sense of the Cauchy-Kowalevski 
Theorem with respect to the space variables, cf. \cite{Hd23}.

\begin{problem} \label{pr.1}
Given $m\geq 1$, functions $ u_j\in C^{2m-j+1, 0} (\overline{\Gamma_T})$,  $1\leq j \leq 2m$,
 and $f \in C(\overline \Omega_T)$ find a function $u \in C^{2m, 1}(\Omega_T )\cap
C^{2m-1, 0}((\Omega \cup \overline \Gamma)_T ) $ satisfying 
\begin{equation} \label{eq.heat}
{\mathcal L}_m u = f \mbox{ in } \Omega_T ,
\end{equation}
\begin{equation} \label{eq.Cauchy}
B_j u(x,t)=u_{j+1}(x, t) \mbox{ on } \overline{\Gamma_T}  
\mbox{ for all } 0\leq j \leq 2m-1.
\end{equation}
\end{problem}

If the hypersurface $\Gamma$ and the data of the problem are real analytic then the 
Cauchy-Kowalevski Theorem implies that problem \eqref{eq.heat}, \eqref{eq.Cauchy} has one and 
only one solution in the class of (even formal) power series. However the theorem does not 
imply the existence of solutions to Problem \ref{pr.1} because it grants the solution in a 
small neighbourhood of the hypersurface  $\Gamma_T$ only (but not in a given domain 
$\Omega_T$!). We emphasize that, unlike the classical case, we do not ask for the 
hypersurface $\Gamma$  or/and  the coefficients of the operators $B_j$ or/and the data $f$ 
or/and $u_j$ to be real analytic. 

Of course, the above trick with the Laplace transform suggests us that the problem 
is equivalent to an ill-posed problem for the strongly elliptic operator $(-\Delta)^m$ in 
$\Omega$ with the Cauchy data on $\Gamma$, i.e. Problem \ref{pr.1} is ill-posed itself, too.

\section{Solvability Conditions} 
\noindent 
 
We begin this section proving that Problem \ref{pr.1} can 
not have more than one solution in the spaces of differentiable 
(non-analytic) functions. 

To investigate Problem \ref{pr.1}, 
we use an integral representation constructed with the use the fundamental solution 
$\Phi_m(x,t)$ to polyharmonic heat operator ${\mathcal L}_m$. If $m=1$ then 
\begin{equation} \label{eq.fund.1}
\Phi_1(x,t)=\begin{cases}
\frac{e^{-\frac{|x|^2}{4 \mu \, t}} }{\left(2\sqrt{\pi \mu \, t }\right)^n } & 
\mbox{ if } t>0,\\ 0 & \mbox{ if } t\leqslant   0,
\end{cases} 
\end{equation} 
 see, for instance, \cite{frid,mih76}. Unfortunately, if $m>1$ then 
the fundamental solution can not be represented as an elementary function, 
see, for instance, \cite[Ch. 2, \S 1]{eid}, \cite{rep+eid}, 
\begin{equation} \label{eq.fund.m}
\Phi_m(x,t)= 
\begin{cases} 
k_{n,m} t^{-n/2m} \int_0^{+\infty} \rho^{n-1} e^{-\rho^{2m}}
\Big(\frac{ |x| \rho}{t^{1/2m}}\Big)^{1-n/2} J_{n/2-1} \Big(\frac{ |x| \rho}{t^{1/2m}}\Big) d\rho & \mbox{ if } t>0, \\
 0 & \mbox{ if } t\leqslant   0,
\end{cases} 
\end{equation}
where  $k_{n,m}$ is a normalization constant and $J_{p}$ is the Bessel function of the first 
kind and of order $p$ (see, for example, \cite[Ch. 5, \S 23]{Vla}). 

The fundamental solution allows to construct a useful integral 
Green formula for the operator ${\mathcal L}_m$. With this purpose, 
 Denote by $\{C_0, \dots C_{2m-1}\}$ the Dirichlet 
system associated with the Dirichlet system $\{B_0, \dots B_{2m-1}\}$ via (first) Green formula
for the operator $\Delta^m$, i.e.
$$
\int_{\partial \Omega} \Big( \sum_{j=0}^{2m-1} 
C_{2m-1-j} v B_j u \Big) ds = 
(\Delta ^m u,v) _{L^2 (\Omega)} - (u, \Delta ^m v) _{L^2 (\Omega) }
$$
for all $u,v \in C^\infty (\overline \Omega)$. For instance, if 
$\{B_0, \dots B_{2m-1}\} =(1,\partial_ \nu, \Delta, \partial_ \nu \Delta , \dots 
\Delta^{m-1}, \partial_ \nu \Delta^{m-1} )$ then $\{C_0, \dots C_{2m-1}\} = 
(1, -\partial_ \nu, \Delta , - \partial_ \nu \Delta ,  ,\dots 
\Delta^{m-1}, -\partial_ \nu \Delta^{m-1} )$. 

Consider the cylinder type domain  
$\Omega_{T_1,T_2} = \Omega_{T_2} \setminus \overline{\Omega_{T_1}}$ with $0\leq T_1 <T_2$  
and a closed measurable set $S \subset \partial \Omega$. For functions 
$f \in L^2 (\Omega_{T_1,T_2})$, $v_j \in L^2 ([0,T], H^{2m-j-1/2}(S_T))$, 
$h \in H^{1/2}(\Omega)$ we introduce the following potentials:  
\begin{equation*} 
I_{\Omega, T_1} (h) (x,t)= \int\limits_{\Omega}\Phi(x-y, t)h(y) dy,  \quad
G_{\Omega,T_1} (f) (x,t)=\int\limits_{T_1}^t\int\limits_\Omega \Phi(x-y,\ t-\tau)f(y, \tau)
dy d\tau, 
\end{equation*}
\begin{equation*} 
V^{(j)}_{S,T_1} (v_j) (x,t)=\int\limits_{T_1}^t\int\limits_{S} C_j 
\Phi_m(x-y, t-\tau) v_j(y, \tau)
ds(y)d\tau, \, \, 0\leq j \leq 2m-1
\end{equation*}
(see, for instance,  \cite[Ch. 1, \S 3 and Ch. 5, \S 2]{frid},  \cite[Ch. 4, \S 1]{lad}, 
\cite[Ch. 3, \S 10]{landis} for $m=1$). 
The potential $I_{\Omega,T_1} (h)$ is an analogue of the \emph{Poisson integral} and 
the function $G_{\Omega,T_1} (f)$ is an analogue of the volume heat potential related to $m=1$.
The functions $V^{(0)}_{S,T_1} (v)$ and $V^{(1)}_{S,T_1} (v)$ are often referred to 
as \emph{single layer heat potential} and \emph{double layer heat  
potential}, respectively, if $m=1$. By the construction, 
all these potentials are (improper) integral depending on the parameters $(x,t)$.

Next, we need the so-called Green formula for the polyharmonic heat operator.

\begin{lemma} \label{l.Green}
For all $0 \leq T_1 < T_2$ and all  $u \in C^{2m,1} (\overline{\Omega_{T_1,T_2}}) $
with  the following  formula holds:
\begin{equation} \label{eq.Green}
\left.
\begin{aligned}
u(x, t) \mbox{ in } \Omega_{T_1,T_2}  \\
0 \mbox{ outside } \overline{\Omega_{T_1,T_2}} 
\end{aligned}
\right\} \! = 
I_{\Omega,T_1} (u)   + G_{\Omega,T_1} ({\mathcal L}_mu)  + \sum_{j=0}^{2m-1}
V ^{(j)}_{\partial \Omega, T_1} \left( 
B_j u \right)  . 
\end{equation}
\end{lemma}

\begin{proof}  See, for instance,  \cite[ch. 6, \S 12]{svesh} for $m=1$ and 
 \cite[theorem 2.4.8]{Tark37} for more general 
operators, admitting fundamental solutions/parametreces).
\end{proof}

Formulas \eqref{eq.fund.1}, \eqref{eq.fund.m} mean that the kernels 
$\Phi_m(x-y,t-\tau)$ are smooth outside 
the diagonal  $\{ (x,t ) =(y,\tau )\}$ and real analytic with respect to 
the space variables. In particular, this means 
that the $2m$-parabolic operator ${\mathcal L}_m$ is hypoelliptic.  
Moreover, 
any $C^{2m,1} (\Omega_{T_1,T_2})$-solution $v$ to the polyharmonic 
heat equation ${\mathcal L}_m v = 0$ in the cylinder domain $\Omega_{T_1, T_2}$ 
belongs to $C^\infty (\Omega_{T_1,T_2})$ and, actually $v (x,t)$
 is real analytic with respect to the space variable  
$x \in \Omega$ for each $t \in (T_1,T_2)$ (for $m=1$, see, for instance, 
\cite[Ch. VI, \S 1, Theorem 1]{mih76} and for $m>1$ see \cite[Ch. 2, \S 1, Theorem 2.1]{eid}. 
Then Green formula 
\eqref{eq.Green} and the information on the kernel $\Phi_m$ 
provide us with a Uniqueness Theorem for Problem \ref{pr.1}.

\begin{theorem}[A Uniqueness Theorem] \label{t.Uniq}
 If $\Gamma$ has at least one interior point in the relative topology of $\partial\Omega$ 
then Problem \ref{pr.1} has no more than one solution.
\end{theorem}

\begin{proof} For $m=1$ see \cite[Theorem 1, Corollary 1]{PuSh12}.
For $m>1$ the proof can be done in the same way with natural modifications. 
Indeed, under the hypothesis of the theorem there is an interior (in the relative topology 
of $\Gamma$!) point $x_0$  on $\Gamma$. Then there is such a number $r>0$ that $B(x_0,\ r)
\cap\partial\Omega\subset\Gamma$ where  $B(x_0,\ r)$ is ball in ${\mathbb R}^n$ with center 
at $x_0$ and radius  $r$. Fix an arbitrary point $(x',t') \in \Omega_T$. Clearly, 
there is a domain $\Omega' \ni x'$ satisfying $\Omega ' \subset \Omega$ and $\Omega ' 
\cap \partial \Omega \subset \Gamma \cap B(x_0,\ r)$. Then $(x', t') \in {\Omega'_{T_1,T_2}}$ 
with some $0 < T_1 < T_2 <T$.

But $u\in C^{2m,1}(\Omega'_{T_1,T_2})\cap C^{2m-1,0}(\overline{\Omega'_{T_1,T_2}})$ 
(for $m=1$ see, for instance,  \cite[Ch. 1, \S 3 and Ch. 5, \S 2]{frid} and 
for $m>1$ it follows from \cite[Ch. 2, \S 1, Theorem 2.2]{eid}) and   
${\mathcal L}_{m}u = 0$ in $\Omega'_{T_1,T_2}$ under the hypothesis of the theorem. Hence 
formula \eqref{eq.Green} implies:  
\begin{equation} \label{2.8a}
\left.
\begin{aligned}
u(x, t),\ (x,t)\in \Omega'_{T_1,T_2}  \\
0,\ (x, t)\not\in \overline{\Omega'_{T_1,T_2}} 
\end{aligned}
\right\}= I_{\Omega',T_1} (u) (x,t) + \sum_{j=0}^{2m-1}
V^{(j)} _{\partial \Omega' \setminus \Gamma, T_1} \left( B_j u\right) 
(x,t), 
\end{equation}
because  $B_j u \equiv  0$ on $\Gamma_T$ for all $0\leq j \leq 2m-1$.

Taking into account the character of the singularity of the kernel 
$\Phi_m (x-y,t-\tau)$ we conclude that the following properties are fulfilled for 
the integrals, depending on parameter, from the right hand side of identity 
\eqref{2.8a}:
\begin{equation*}
 I_{\Omega',T_1} (u)  \in C^{2m,1} (\{x \in {\mathbb R}^n,T_1<t<T_2\}), 
\end{equation*}
\begin{equation*}
 V ^{(j)}_{\partial \Omega' \setminus \Gamma, T_1<t<T_2} \left( 
B_j u \right) 
\in C^{2m,1} (\{x \in {\mathbb R}^n \setminus (\partial \Omega' \setminus \Gamma), 
T_1<t<T_2 \}) 
\end{equation*}
(see, for instance, \cite[Ch. 1, \S 3 and  Ch. 5, \S 2]{frid}, 
\cite[Ch. 4, \S 1]{lad} or \cite[Ch. 3, \S 10]{landis}  for m=1).
Moreover, as $\Phi_m$ is a fundamental solution to the polyharmonic heat operator 
then  
\begin{equation*}
{\mathcal L} _{m} (x,t) \Phi_m (x-y,t-\tau) = 0 \mbox{ for } (x,t) \ne (y,\tau),
\end{equation*}
and therefore, using Leibniz rule for differentiation of integrals 
depending on parameter we obtain: 
\begin{equation*}
{\mathcal L} _{m} I_{\Omega',T_1} (u) = 0  \mbox{ in the domain }   
\{x \in {\mathbb R}^n, T_1<t<T_2\}, 
\end{equation*}
\begin{equation*}
{\mathcal L} _{m} V^{(j)} _{\partial \Omega' \setminus \Gamma, T_1} \left( B_j u\right) 
=  0 \mbox{ in  }
\Omega''_{T_1,T_2} = \{x \in {\mathbb R}^n \setminus 
(\partial \Omega' \setminus \Gamma), T_1<t<T_2\} \mbox{ for all } 0\leq j \leq 2m-1.
\end{equation*}
Hence the function  
\begin{equation*}
v(x,t) = I_{\Omega',T_1} (u) (x,t) + 
 V^{(j)} _{\partial \Omega' \setminus \Gamma, T_1} \left( B_j u\right) 
(x,t)   , 
\end{equation*}
satisfies the polyharmonic heat equation 
$ ({\mathcal L} _{m} v) (x,t) = 0$ in $\Omega''_{T_1,T_2}$. 
As we mentioned above, this implies that the function $v (x,t)$ is real analytic with respect 
to the space variable $x \in {\mathbb R}^n \setminus (\partial \Omega' \setminus \Gamma)$ for 
any $T_1<t<T_2$ . By the construction the function $v (x,t)$ is real analytic 
with respect to $x$ in the ball $B(x_0,r)$ and it equals to zero for $x \in B(x_0,R) 
\setminus \overline \Omega$ for all $T_1<t<T_2$. Therefore, the Uniqueness Theorem for real 
analytic functions yields $v(x,t) \equiv 0$ in $\Omega''_{T_1,T_2}$, and in the cylinder 
$\Omega'_{T_1,T_2}$, the containing point  $(x',t')$. Now it follows from \eqref{2.8a}) that  
$u (x',t') =v(x',t') = 0$ and then, since the point $(x',t') \in \Omega_T$ is arbitrary we 
conclude that $u \equiv 0$ in $\Omega_T$.  
\end{proof}

Now we are ready to formulate a solvabilty criterion for Problem \ref{pr.1}. 
As before, we assume that  $\Gamma$ is a relatively open connected set of $\partial \Omega$. 
Then we may find a set $\Omega^+ \subset {\mathbb R}^n$ in such a way that the set $D=\Omega\cup\Gamma\cup\Omega^+$ 
would be a bounded domain with piece-wise smooth boundary. It is convenient to set 
$\Omega^- = \Omega$. For a function 
$v$ on $D_T$ we denote by  $v^+$ its restriction to $\Omega^+_T$ and, similarly, we 
denote by  $v^-$ its restriction to $\Omega_T$. It is natural to denote limit values of 
 $v^\pm$ on $\Gamma_T$, when they are defined, by  $v^\pm_{|\Gamma_T}$. 
Actually, for $m=1$ similar solvability criterions for Problem \ref{pr.1} were obtained in 
\cite{PuSh12} and \cite{KuSh2019}.

\begin{theorem}[Solvability criterion] \label{t.sol}
Let $\lambda\in (0,1)$, $\partial \Omega$ belong  to $C^{2m-1+\lambda}$ and let  
$\Gamma$ be a relatively open connected subset of $\partial \Omega$. If 
$f\in C^{0,0,\lambda,\lambda/2}
(\overline{\Omega_T})$,   $u_j\in C^{2m-j,0,\lambda,\lambda/2}(\overline{\Gamma_T})$, 
 $1\leq j \leq 2m$, then
Problem \eqref{eq.heat}, \eqref{eq.Cauchy}  
is solvable in the space $C^{2m,1,\lambda,\lambda/2}(\Omega_T)\cap C^{2m-1,0,\lambda,\lambda/2}
(\Omega_T \cup \Gamma_T)$ if and only if there is a function $F\in C^\infty (D_T)$ 
satisfying the following two conditions:
1) ${\mathcal L}_m F=0$ in $D_T$, 2) $F=G_{\Omega, 0} (f)+ 
\sum_{j=0}^{2m-1} V^{(j)}_{\overline{\Gamma},0} 
(u_{j+1}) $ in $\Omega^+_T$.
\end{theorem}

{\it Proof.}\emph{ \textrm{Necessity.}} Let a function  $u(x, t)\in C^{2m,1,\lambda,
\lambda/2}(\Omega_T)\cap C^{2m-1,0,\lambda,\lambda/2}(\Omega_T\cup\Gamma_T)$ 
satisfies \eqref{eq.heat}, \eqref{eq.Cauchy}.  
Clearly, the function  $u(x,t)$ belongs to the space $C^{2m,1,\lambda,\lambda/2}
(\Omega'_T) \cap C^{2m-1,0,\lambda,\lambda/2} (\overline{\Omega'_T}$ for each cylindrical domain $\Omega'_T$ with such a base $\Omega'$ 
that $\Omega' \subset \Omega$ and $\overline{\Omega'}\cap\partial \Omega \subset \Gamma$. 
Besides, ${\mathcal L}u =f \in C^{0,0,\lambda,\lambda/2}(\overline{\Omega_T'})$. 
Without loss of the generality we may assume that the interior part $\Gamma' $ of the set  
$\overline{\Omega'}\cap\partial \Omega$ is non-empty. Consider in the domain $D_T$ the 
functions
\begin{equation} \label{eq.FF}
\mathcal{F}=G_{\Omega, 0} (f)+\sum_{j=0}^{2m-1} V^{(j)}_{\overline{\Gamma},0} 
(u_{j+1})\mbox{ and } F={\mathcal F}-\chi_{\Omega_T} u,
\end{equation} 
where $\chi_{M}$ is a characteristic function of the set 
$M \subset {\mathbb R}^{n+1}$. By the very construction condition 2) is fulfilled for it. 
Note that  $\chi_{\Omega_T} u=\chi_{\Omega'_T} u$ in $D'_T$, where $D'= \Omega' \cup 
\Gamma' \cup \Omega^+$. Then Lemma \ref{l.Green} yields
\begin{equation} \label{eq.F}
F=G_{\Omega \setminus \overline{\Omega'}, 0} (f)+
\sum_{j=0}^{2m-1} V^{(j)}_{\overline{\Gamma},0} 
(u_{j+1}) - I_{\Omega',0} (u) \mbox{ in } D'_T.
\end{equation} 

Arguing as in the proof of Theorem  \ref{t.Uniq} we conclude that  
each of the integrals in the right hand side of \eqref{eq.F} is smooth 
outside the corresponding integration set and each satisfies 
homogeneous polyharmonic heat equation there. In particular, we see that 
$F \in C^\infty (D'_T)$ and ${\mathcal L} F =0$ in $D'_T$ because of 
\cite[Ch. VI, \S 1, Theorem 1]{mih76}. Obviously, for 
any point $(x, t)\in D_T$ there is a domain $D'_T$ containing $(x,\ t)$. That is why  
${\mathcal L}_m F =0$ in $D_T$, and hence $F$ belongs to the space $ C^\infty(D_T)$. 
Thus, this function satisfies condition 1), too. 

\emph{\textrm{Sufficiency.}} Let there be a function $F \in  C^{\infty}(D_T) $, satisfying 
conditions 1) and 2) of the theorem. Consider on the set $D_T$ the function 
\begin{equation} \label{eq.sol}
U=\mathcal{F}- F .
\end{equation} 
As  $f \in C^{0,0,\lambda,\lambda/2}
(\overline{\Omega_T})$ then the results of \cite[Ch. 1, \S 3]{frid}, 
\cite[Ch. 4, \S\S 11-14]{lad} for $m=1$ and \cite[Ch. 2, \S 1, Theorem 2.2]{eid} 
for $m>1$ imply 
\begin{equation} \label{eq.G1} 
G_{\Omega, 0} (f) \in C^{2m,1,\lambda,\lambda/2}(\overline{\Omega^\pm_T}) 
\cap  C^{2m-1,0,\lambda,\lambda/2}(D_T)  
\end{equation} 
and, moreover, 
\begin{equation} \label{eq.G2}  
{\mathcal L}_m G^-_{\Omega, 0} (f) = f \mbox{ in } \Omega_T, \quad 
{\mathcal L}_m G^+_{\Omega, 0} (f) = 0 \mbox{ in }  \Omega^+_T.
\end{equation} 
Since  $u_j \in C^{2m-j,0,\lambda,\lambda/2} 
(\overline{\Gamma_T})$ then the results of  \cite[Ch. 4, \S\S 11-14]{lad}, 
\cite[Ch. 5, \S 2]{frid}
 for $m=1$ and \cite[Ch. 2, \S 1, Theorem 2.2]{eid} 
for $m>1$ yield
\begin{equation} \label{eq.Vj}  
V^{(j)}_{\overline \Gamma, 0} (u_j) \in C^\infty(\Omega^\pm_T) \cap 
C^{2m-1,0,\lambda,\lambda/2}((\Omega^\pm\cup \Gamma)_T), \quad 
{\mathcal L}^{(j)} V_{\overline \Gamma, 0} (u_j) = 0 \mbox{ in } \Omega_T \cup \Omega^+_T .
\end{equation} 

Since $F \in C^{\infty} (D_T) \subset C^{1,0,\lambda,\lambda/2}((\Omega^+\cup \Gamma)_T)$ 
then formulas \eqref{eq.sol}--\eqref{eq.Vj} imply that  
$U $ belongs $ C^{2m,1,\lambda,\lambda/2}(\Omega^\pm_T )  
\cap C^{2m-1,0,\lambda,\lambda/2}((\Omega^\pm \cup \Gamma)_T)$ and 
${\mathcal L} U = \chi_{D_T} f$  in $\Omega_T \cup \Omega^+_T$. 
In particular, \eqref{eq.heat} is fulfilled for $U^-$.
Let us show that the function $U^-$ satisfies \eqref{eq.Cauchy}. 
Since  $F \in  C^{\infty} (D_T)$ we see that $\partial ^\alpha F^- = \partial ^\alpha F^+$ 
on $\Gamma_T$ for $\alpha \in {\mathbb Z}_+$ with $|\alpha|\leqslant   2m-1$ and 
\begin{equation*} 
\partial ^\alpha F^+_{|\Gamma_T} = \left( \partial ^\alpha  G_{\Omega, 0} (f)+ 
\sum_{j=0}^{2m-1} \partial ^\alpha  (V^{(j)}_{\overline{\Gamma},0} (u_{j+1}) \right)^+
_{|\Gamma_T}.
\end{equation*}  
Thus, it follows from formula \eqref{eq.G1} that for all $0\leq i \leq 2m-1$ we have
\begin{equation} \label{eq.BoundPot} 
B_i U^-_{|\Gamma_T}=
\Big(B_i\Big(\sum_{j=0}^{2m-1} V^{(j)}_{\overline{\Gamma},0} (u_{j+1})\Big)\Big)^-_{|\Gamma_T} - \Big(B_i \Big(\sum_
{j=0}^{2m-1} V^{(j)}_{\overline{\Gamma},0} (u_{j+1})\Big)\Big)^+_{|\Gamma_T} .
\end{equation}

Hence, in order to finish the proof we need the following lemma. 

\begin{lemma} \label{l.W.jump} Let $\Gamma \in C^{2m-1+\lambda}$ and $u_j \in 
C ^{2m-j,0,\lambda ,\lambda/2} (\overline{\Gamma_T})$, $1\leq j \leq 2m$.
Then   
\begin{equation} \label{eq.Wnu} 
\Big(B_i\Big(\sum_{j=0}^{2m-1} V^{(j)}_{\overline{\Gamma},0} (u_{j+1})\Big)\Big)^-_{|\Gamma_T} 
- \Big(B_i \Big(\sum_
{j=0}^{2m-1} V^{(j)}_{\overline{\Gamma},0} (u_{j+1})\Big)\Big)^+_{|\Gamma_T} = u_{i+1}, 
\,\, 0\leq i \leq 2m-1. 
\end{equation}  
\end{lemma}

{\it Proof.} It is similar to the proof of the analogous lemmata for 
the heat Single and Double Layer Potentials 
(see, for instance, \cite[lemma 3]{PuSh12}, \cite[Ch. 3, \S 10, theorem 10.1]{landis} for 
$m=1$ and a different function class or \cite[Lemma 2.7]{ShTaLMS} for elliptic 
potentials).
\hfill $\square$

Using lemma \ref{l.W.jump} and formulas \eqref{eq.G1}, \eqref{eq.BoundPot}, 
we conclude that  $B_j U^-_{|\Gamma_T}= u_{j+1}$ for all $0\leq j \leq 2m-1$,
i.e. the second equation in \eqref{eq.Cauchy} 
is fulfilled for $U^-$. 
Thus, function $ u (x,t)=U^-(x, t)$ satisfies conditions \eqref{eq.heat}, \eqref{eq.Cauchy}. 
The proof is complete. \hfill $\square$

We note that  Theorem \ref{t.sol} is also an analogue
of Theorem by Aizenberg and Kytmanov 
\cite{AKy} describing solvability conditions of the Cauchy problem for the Cauchy--Riemann 
system (cf. also \cite{Sh1} in the Cauchy Problem for Laplace Equation or  
 \cite{Tark36} in the  Cauchy problem for general elliptic systems).

We note also that formula \eqref{eq.sol}, obtained in the proof of Theorem  \ref{t.sol}, gives 
the unique solution to Problem  \ref{pr.1}. Clearly, if we will be able to write  
the extension $F$ of the sum of potentials $G_{\Omega, 0} (f)+
\sum_{j=0}^{2m-1}V^{(j)}_{\overline{\Gamma},0} (u_{j+1})$ 
from $\Omega^+_T$ onto $D_T$ as a series with respect to 
special functions or a limit of parameter depending integrals then we will get  
Carleman's type formula for solutions to Problem \ref{pr.1} (cf. \cite{AKy}). 
However, for the best way for this purpose is to use the Fourier series 
in the framework of the Hilbert space theory, see \cite{VKuSh2022}. Unfortunately, 
this is not a short story because one needs approximation theorems 
in spaces of solutions to the homogeneous polyharmonic heat equation that we are 
not ready to prove right now. Thus we finish 
our paper with a statement extending Theorem \ref{t.sol} to the anisotropic Sobolev spaces, 
leaving the construction of the Carleman's type formulae for the next article.

First of all, we need the following lemma.

\begin{lemma} \label{l.ext}
Let $\partial \Omega \in C^{2m+1}$ and let $\Gamma$ be a relatively open subset 
of $\partial \Omega$ with boundary $\partial \Gamma \in C^{2m+\lambda}$. 
If $u_j \in C^{2m+1-j,0,\lambda,\lambda/2} (\overline {\Gamma_T})$, 
$1\leq j \leq 2m$,  
then there exist  functions 
$\tilde u_j\in  C^{2m+1-j,0,\lambda,\lambda/2} (\partial \Omega_T)$ such that 
$\tilde u_j =u_j $ on $\overline {\Gamma_T}$, $1\leq j \leq 2m$, and a function 
$\tilde u \in C^{2m,1,\lambda,\lambda/2} (\overline{\Omega_T}) $
such that $B_j \tilde u =\tilde u_{j+1} $ on $(\partial \Omega)_T$ 
for all $0\leq j \leq 2m-1$. 
\end{lemma}

\begin{proof} 
We may adopt the standard arguments from 
\cite[Lemma 6.37]{GiTru83} related to isotropic spaces. Indeed, according to it, under 
our assumptions, for any $s\leq 2m$  and any 
 $v\in C^{s,\lambda} (\overline \Gamma)$  there is 
$\tilde v_j\in  C^{s,\lambda} (\partial \Omega)$ such that $v =v_0 $ on 
$\overline {\Gamma}$. 
The construction of the extension involves the rectifying diffeomorphism of $\partial  \Gamma$ 
and a suitable partition of unity of a neighbourhood of $\partial  \Gamma$, only. 
Thus, we conclude there are functions 
$\tilde u_j\in  C^{2m-j+1,0,\lambda,\lambda/2} (\partial \Omega_T)$ 
such that $\tilde u_j =u_j $ on $\overline {\Gamma_T}$, $1\leq j \leq 2m$. 

Next, we use the existence of the Poisson kernel $P_{\Delta^{2m},\Omega} (x,y)$ for the 
Dirichlet problem related to the operator $\Delta^{2m}$, see \cite{ArCrLi83}. It is known that 
the problem is well-posed over the scale of H\"older spaces in $\Omega$. Namely, 
if $\partial \Omega\in C^{s+1,\lambda}$, $s\geq 2m-1$, then for each  
 $\oplus_{j=0}^{2m-1} v_j \in C^{s-j,\lambda} (\partial \Omega)$ the integral 
\begin{equation*}
v(x) = {\mathcal P}_{\Delta^{2m},\Omega}  (\oplus_{j=0}^{2m-1} v_j) (x) 
= \int_{\partial\Omega} \Big(\sum_{j=0}^{2m-1} (B_j (y) P_{\Delta^{2m},\Omega} ) (x,y) v_j(y) 
\Big) ds (y)
\end{equation*}
belongs to $C^{s,\lambda} (\overline \Omega)$ and satisfies
$\Delta ^{2m} v =0 $ in $\Omega$ and $B_jv=v_j$  on 
$ \partial\Omega$ for all $0\leq j\leq 2m-1$. 

Now,  we set 
\begin{equation*}
\tilde u_0  (x)= {\mathcal P}_{\Delta^{2m},\Omega}   (\oplus_{j=0}^{2m-1} \tilde u_{j+1}) 
(\cdot,0) 
(x) \in C^{2m-1,\lambda} (\overline \Omega) \cap C^{2m,\lambda} (\Omega).
\end{equation*}

Now,  we may take as $\tilde u  (x,t) \in C^{2m,1,\lambda,\lambda/2} (\Omega_T) \cap C^{2m-1,
0,\lambda,\lambda/2}   (\overline{\Omega_T)}$ the unique solution to the parabolic initial  
boundary problem 
\begin{equation*}
\left\{
\begin{array}{lll}
\partial _t \tilde u  (x,t) + \Delta^{2m} \tilde u  (x,t) = 0 & \mbox{ in } & \Omega_T ,\\ 
\oplus_{j=0}^{2m-1}B_j \tilde u  (x,t) = 
\oplus_{j=0}^{2m-1} \tilde u _{j+1} (x,t)   & \mbox{ on } & (\partial \Omega)_T ,\\
\tilde u  (x,0) = \tilde u _0 (x)  & \mbox{ on } & \overline \Omega ,\\
\end{array}
\right.
\end{equation*}
see, for instance, \cite[Ch.5, \S 6]{lad} for $m=1$ or \cite[Ch. 3, \S 1]{eid} for $m\geq1$. 
But of course, there are other possibilities 
to choose a function $\tilde u$ with the desired properties. 
\end{proof}  

Under the assumptions of Lemma \ref{l.ext}, we set 
\begin{equation} \label{eq.tildeF}
\tilde {\mathcal F}=G_{\Omega, 0} (f)+\sum_{j=0}^{2m-1}
V^{(j)}_{\partial \Omega,0} (\tilde u_{j+1})  + 
I_{\Omega,0} (\tilde u).
\end{equation}

\begin{corollary} \label{c.sol.n}
Let $\lambda\in (0,1)$, $\partial \Omega$ belong  to $C^{2m+1+\lambda}$ and let  
$\Gamma$ be a relatively open connected subset of $\partial \Omega$ 
with boundary $\partial \Gamma \in C^{2m+\lambda}$. If 
$f\in C^{0,0,\lambda,\lambda/2}
(\overline{\Omega_T})$,   $u_j\in C^{2m-j+1,0,\lambda,\lambda/2}(\overline{\Gamma_T})$, 
then Problem \eqref{eq.heat}, \eqref{eq.Cauchy}  
is solvable in the space $C^{2m,1,\lambda,\lambda/2}(\Omega_T)\cap C^{2m-1,0,
\lambda,\lambda/2}(\Omega_T \cup \Gamma_T) \cap H^{2m,1}(\Omega_T)$ if and only if 
there is a function $\tilde F\in C^\infty (D_T) \cap H^{2m,1}(D_T)$ 
satisfying the following two conditions:
1') ${\mathcal L} \tilde F=0$ in $D_T$, 
2') $\tilde F=\tilde{\mathcal F}$ in $\Omega^+_T$.
\end{corollary}

\begin{proof} First of all, we note that, by Green formula (\ref{eq.Green}), 
we have  $\tilde{\mathcal F} = G_{\Omega, 0} (f-{\mathcal L}\tilde u) + \chi_{\Omega_T} \tilde u$ and 
then $\tilde{\mathcal F} \in C^{2m,1,\lambda,\lambda/2} (\overline{\Omega^\pm_T})$ 
because of (\ref{eq.G1}). On the other hand, 
\begin{equation} \label{eq.difference}
\tilde{\mathcal F} -{\mathcal F} = 
\sum_{j=0}^{2m-1}
V^{(j)}_{\partial \Omega \setminus \Gamma ,0} (\tilde u_{j+1})
 + I_{\Omega,0} (\tilde u).
\end{equation}
This means that the function $\tilde{\mathcal F} -{\mathcal F}$ satisfies the 
${\mathcal L} (\tilde{\mathcal F} -{\mathcal F}) =0$ in $D_T$ and 
hence the function ${\mathcal F}$ extends to $D_T$ as a solution of the 
heat equation if and only if  function $\tilde {\mathcal F}$ extends to $D_T$
as a solution of the polyharmonic heat equation, too. 

Let Problem \eqref{eq.heat}, \eqref{eq.Cauchy} be solvable in the space 
$C^{2m,1,\lambda,\lambda/2}(\Omega_T)\cap C^{2m-1,0,
\lambda,\lambda/2}(\Omega_T \cup \Gamma_T) \cap H^{2m,1}(\Omega_T)$. Then 
formulas \eqref{eq.FF} and \eqref{eq.difference} imply 
\begin{equation*}
\tilde F = \tilde{\mathcal F} -\chi_{\Omega_T} u \in H^{2m,1} (\Omega^\pm_T)
\mbox{ and } {\mathcal L} \tilde F =0 \mbox{ in } D_T.
\end{equation*}
Now, as $\tilde F \in H^{2m,1} (\Omega^\pm_T)\cap C^\infty (D_T)$ (see 
\cite[Ch. VI, \S 1, Theorem 1]{mih76}) we conclude that $\tilde F \in H^{2m,1} (D_T)$, 
i.e. conditions 1'), 2') of the corollary are fulfilled.

If conditions 1'), 2') of the corollary hold true then 
conditions 1), 2) of Theorem \ref{t.sol} are fulfilled, too. 
Moreover, formulas \eqref{eq.sol} and 
\eqref{eq.difference} imply that in $D_T$ we have
\begin{equation} \label{eq.sol.tilde}
U= {\mathcal F} - F = \tilde {\mathcal F} - \tilde F \in 
H^{2m,1} (\Omega^\pm_T)
\end{equation}
and the $U^-$ is the solution to Problem \ref{pr.1} in the space  
$C^{2m,1,\lambda,\lambda/2}(\Omega_T)\cap C^{2m-1,0,
\lambda,\lambda/2}(\Omega_T \cup \Gamma_T) \cap H^{2m,1} (\Omega^\pm_T)$ 
by Theorem \ref{t.sol}. 
\end{proof}

\smallskip

\textit{Acknowledgments\,}
The second author was supported by the Russian Science Foundation,  
grant N 20-11-20117.

\end{document}